\documentclass{article}
\usepackage{authblk}
\usepackage{bookmark}
\usepackage[a4paper, margin = 3.5cm]{geometry} 
\usepackage{amsthm, amssymb, latexsym, graphics}
\usepackage{amsmath,amsfonts}

\date{ }

\providecommand\bnabla{\ensuremath{\nabla}}
\providecommand\bcdot{\ensuremath{\cdot}}
\providecommand\bu{\ensuremath{u}}
\providecommand\bcolon{\ensuremath{\colon}}
\providecommand\rmDelta{\ensuremath{\Delta}}

\author[1]{Fabian Bleitner}
\author[2]{Camilla Nobili}

\affil[1]{\small{University of Hamburg, Department of Mathematics, Bundesstraße 55, 20146 Hamburg, Germany}}
\affil[2]{\small{University of Surrey, School of Mathematics and Physics, Guildford, Surrey GU2 7XH, United Kingdom}}

\usepackage{graphicx}
\usepackage{newtxtext}
\usepackage{newtxmath}
\usepackage{natbib}
\usepackage{hyperref}
\hypersetup{
    colorlinks = true,
    urlcolor   = blue,
    citecolor  = black,
}
\newtheorem{lemma}{Lemma}

\newtheorem{remark}{Remark}
\newcommand{\RomanNumeralCaps}[1]
\linenumbers

\usepackage{amsmath}
\newcommand{\Nu}{\mbox{\textit{Nu}}}    
\newcommand{\Ra}{\mbox{\textit{Ra}}}    
\newcommand{\Pra}{\mbox{\textit{Pr}}}   
\renewcommand{\Pr}{\Pra}
\newtheorem{theorem}{Theorem}

\usepackage{xcolor}
\newcommand{\colorBig}{teal}
\newcommand{\colorBoth}{black}
\newcommand{\colorSmall}{violet}

\usepackage{newtxtext}
\usepackage{newtxmath}
\usepackage{natbib}
\usepackage{hyperref}
\usepackage{multirow}

\usepackage{nowidow} 

\usepackage{mathtools}

\title{Scaling laws for Rayleigh-B\'enard convection between Navier-slip boundaries}

\begin{document}
\maketitle

\begin{abstract}
We consider the two-dimensional Rayeigh-B\'enard convection problem between Navier-slip fixed-temperature boundary conditions and present a new upper bound for the Nusselt number. 
The result, based on a localization principle for the Nusselt number and an interpolation bound, exploits the regularity of the flow. On one hand our method yields a shorter proof of the celebrated result in \cite{whiteheadDoeringUltimateState} in the case of free-slip boundary conditions. On the other hand, its combination with a new, refined estimate for the pressure gives a substantial improvement of the interpolation bounds in \cite{drivasNguyenNobiliBoundsOnHeatFluxForRayleighBenardConvectionBetweenNavierSlipFixedTemperatureBoundaries} for slippery boundaries. A rich description of the scaling behaviour arises from our result: depending on the magnitude of the Prandtl number and slip-length, our upper bounds indicate five possible scaling laws: $\Nu\sim (L_s^{-1}\Ra)^{\frac{1}{3}}$, $\Nu\sim (L_s^{-\frac 25}\Ra)^{\frac{5}{13}}$, $\Nu\sim \Ra^{\frac{5}{12}}$, $\Nu\sim \Pr^{-\frac 16}(L_s^{-\frac 43}\Ra)^{\frac{1}
{2}}$ and $\Nu\sim \Pr^{-\frac 16}(L_s^{-\frac 13}\Ra)^{\frac{1}
{2}}$.
\end{abstract}



\section{Introduction}
In this paper we consider a layer of fluid trapped between two parallel horizontal plates held at different temperatures.
{
The dimensionless equations of motions for the Boussinesq approximation are 
\begin{align}
    \frac{1}{\Pra}(\partial_t \bu + \bu \bcdot \bnabla \bu) - \rmDelta \bu +\bnabla p\label{navier-stokes}
    &=
    \Ra T \boldsymbol{e}_2
    \\
    \bnabla \bcdot \bu &= 0  \label{incomp}
    \\
    \partial_t T + \bu\bcdot \bnabla T -\rmDelta T  \label{eq:T}
    &=0\,,
\end{align}
where the Rayleigh number, $\Ra$, is defined as
\begin{align}
    \Ra=\frac{g\alpha \delta T h^3}{\kappa \nu}\label{Ra_def}
\end{align}
and the Prandtl number, $\Pr$, is 
\begin{align}
    \Pra=\frac{\nu}{\kappa}.\label{Pr_def}
\end{align}
In these definitions $g$ is the gravitational constant, $\alpha$ the thermal expansion coefficient, $\nu$ the kinematic viscosity, $\kappa$ the thermal diffusivity, $h$ the {distance between the plates} and $\delta T =T_{\rm{bottom}}-T_{\rm{top}}$ the temperature gap. 

Lengths are measured in
units of $h$, time in units of $\frac{h^2}{\kappa}$, and temperature in units of $\delta T$. In the rectangular domain $\Omega=[0,\Gamma]\times[0,1]$ the velocity $\bu=u_1(\mathbf{x},t)\boldsymbol{e}_1+u_2(\mathbf{x},t)\boldsymbol{e}_2$ and temperature $T=T(\mathbf{x},t)$ are initialized at $t=0$, where
\begin{align}
    \bu(\mathbf{x},0) &= \bu_0(\mathbf{x})
    \\
    T(\mathbf{x},0) &= T_0(\mathbf{x}).
\end{align}
The boundary conditions for the temperature are
\begin{equation}
    \label{temperature_bc}
    \begin{aligned}
    T=0 & \quad \mbox{ at } x_2=1
    \\
    T=1 & \quad \mbox{ at }  x_2=0    
    \end{aligned}
\end{equation}
while we assume Navier-slip boundary conditions for the velocity field, that is
\begin{equation}
    \label{nav-slip}
    \begin{aligned}
        u_2=0,\quad &\partial_2 u_1=-\frac{1}{L_s}u_1\quad \mbox{at } x_2=1      
        \\
        u_2=0,\quad &\partial_2 u_1=\frac{1}{L_s}u_1\phantom{-}\quad \mbox{at } x_2=0,
    \end{aligned}
\end{equation}
where $L_s$ is the constant slip-length. Here $x_2=\mathbf{x}\cdot \boldsymbol{e_2}$.
In the horizontal variable $x_1=\mathbf{x}\cdot \boldsymbol{e_1}$ all variables, including the pressure $p=p(\mathbf{x},t)$, are periodic.}

We are interested in quantifying the heat transport in the upward direction as measured by the non-dimensional Nusselt number 
\begin{equation}\label{Nusselt}
    \Nu=\left\langle \int_0^1(u_2T-\partial_2T)\,dx_2\right\rangle 
\end{equation}
where 
\begin{equation}\label{long-time-average}
    \left\langle \cdot\right\rangle=\limsup_{t\rightarrow \infty}\frac 1t\int_0^t\frac{1}{\Gamma}\int_0^{\Gamma}(\cdot)\, dx_1\, ds.
\end{equation}
This number, of utmost relevance in geophysics and industrial applications \citep{plumley2019scaling}, is predicted to obey a power law scaling of the type 
\begin{equation}
    \Nu\sim \Ra^{\alpha}\Pra^{\beta}.    
\end{equation}
Although physical arguments suggest certain scaling exponents $\alpha$ and $\beta$ (\cite{malkus1954heat, kraichnan1962turbulent, spiegel1971convection, siggia1994high}) and where transitions between scalings could occur \citep{ahlers2009heat}, these theories need to be validated.
While experiments are expensive and difficult \citep{ahlers2006experiments}, numerical studies are limited by the lack of computational power in reaching high-Rayleigh number regimes \citep{plumley2019scaling}.
Even in two spatial dimensions there have been recent debates about the presence/absence of evidence of the "ultimate scaling" (scaling that holds in the regime of very large Rayleigh numbers) for the Nusselt number \citep{zhu2018transition, doering2019absence, zhu2019zhu, doering2020absence}.
We use mathematical analysis in order to derive universal upper bounds for the Nusselt number, which serve as a rigorous indication for the scaling exponents holding in the turbulent regime $\Ra\rightarrow \infty$. The properties of boundary layers and their thickness play a central role in the scaling laws for the Nusselt number in Rayleigh-B\'enard convection (see \cite{nobili2023role} and references therein). For this reason, it is interesting to study how heat transport properties change when varying the boundary conditions. In particular, we ask the following question: are there boundary conditions inhibiting or enhancing heat transport compared to the classical \textit{no-slip} boundary conditions? While many theoretical studies focus on no-slip boundary conditions  \citep{doering1996variational, doering2001upper, DoeringOttoReznikoff2006bounds, otto2011rayleigh, goluskin2016bounds, tobasco2017optimal}, other reasonable boundary conditions have been far less explored. In the early 2000, \cite{ierley2006infinite} considered the Rayleigh-B\'enard convection problem at infinite Prandtl number and free-slip boundary conditions; {their computational result $\Nu\lesssim \Ra^{\frac{5}{12}}$ is obtained by combining the Busse's asymptotic expansion in multi-boundary-layers solutions (multi$-\alpha$ solutions) and the Constantin \& Doering \textit{background field method} approach (\cite{doering1996variational, doering1994variational}).}
Inspired by this result and a numerical study in Otero's thesis \citep{otero2002bounds}, Doering \& Whitehead rigorously proved 
\begin{equation}
    \Nu\lesssim \Ra^{\frac{5}{12}},
\end{equation}
for the two dimensional, finite Prandtl model 
(\cite{whiteheadDoeringUltimateState}) and for the three dimensional, infinite Prandtl number model (\cite{DoeringWhitehead12rigid}) using an elaborate application of the background field method. By a perturbation argument around Stokes equations, \cite{WangWhitehead13} proved 
\begin{equation}
    \Nu\lesssim \Ra^{\frac{5}{12}}+\mathrm{Gr}^2\Ra^{\frac 14}    
\end{equation}
in three-dimensions and small Grashof number ($\mathrm{Gr}=\frac{\Ra}{\Pr}$).

In many physical situations the Navier-slip boundary conditions are used to describe the presence of slip at the solid–liquid interface \citep{uthe2022optical, Neto_2005}.
We notice that, for any finite $L_s> 0$, these conditions imply vorticity production at the boundary. In the limit of infinite slip-length, the Navier-slip boundary conditions reduce to free-slip boundary conditions, while in the limit $L_s\rightarrow 0$ they converge to the no-slip boundary conditions.

Inspired by the seminal paper \cite{whiteheadDoeringUltimateState}, Drivas, Nguyen and the second author of this paper considered the two-dimensional Rayleigh-B\'enard convection model with Navier-slip boundary conditions and rigorously proved the upper bound
\begin{equation}\label{DNN22}
    \Nu\lesssim \Ra^{\frac{5}{12}}+L_s^{-2}\Ra^{\frac 12}
\end{equation}
when $\Pr\geq \Ra^{\frac 34}L_s^{-1}$ and $L_s\gtrsim 1$ \citep{drivasNguyenNobiliBoundsOnHeatFluxForRayleighBenardConvectionBetweenNavierSlipFixedTemperatureBoundaries}. The authors refer to \eqref{DNN22} as \textit{interpolation bound}: assuming $L_s\sim \Ra^{\alpha}$ with $\alpha\geq 0$ (the justification of this choice can be found in the appendix of \cite{bleitnerNobili24}), it can be easily deduced that
\begin{equation}
    \Nu\lesssim \begin{cases}\Ra^{\frac{5}{12}} & \mbox{if } \alpha\geq \frac{1}{24}\\\Ra^{\frac{1}{2}-2\alpha} & \mbox{if } 0\leq\alpha\leq \frac{1}{24}\,.\end{cases}
\end{equation}
In \cite{bleitnerNobili24}, the authors generalized the result in \cite{drivasNguyenNobiliBoundsOnHeatFluxForRayleighBenardConvectionBetweenNavierSlipFixedTemperatureBoundaries} to the case of rough walls and Navier-slip boundary conditions, proving interpolation bounds exhibiting explicit dependency on the spatially varying friction coefficient and curvature.

Given the model \eqref{navier-stokes},\eqref{incomp},\eqref{eq:T} with \eqref{temperature_bc} 
and \eqref{nav-slip}, the objective of this paper is twofold: on one hand we want to apply the so called \textit{direct method} (as outlined in \cite{otto2011rayleigh}) to derive upper bounds on the Nusselt number exhibiting a transparent relation with the thermal boundary layer's thickness. 
For discussions about different approaches to derive bounds on the Nusselt number we refer the readers to \cite{chernyshenko2022relationship} and references therein. On the other hand, we aim at improving the bound in \cite{drivasNguyenNobiliBoundsOnHeatFluxForRayleighBenardConvectionBetweenNavierSlipFixedTemperatureBoundaries} by refining the  estimates on the pressure.
Our new result is stated in the following
{
\begin{theorem}\label{Main-th}
Suppose $\bu_0\in W^{1,4}$ and $0\leq T_0\leq 1$.

If $L_s=\infty$ (i.e. $\mathbf{u}$ satisfies free-slip boundary conditions), then
\begin{equation}\label{standard}
    \Nu \lesssim \Ra^{\frac{5}{12}}\,.
\end{equation}

If $1\leq L_s<\infty$, then 
\begin{align}\label{b1Nu}
    \Nu \lesssim \Ra^\frac{5}{12}+L_s^{-\frac{1}{6}}\Pr^{-\frac{1}{6}}\Ra^\frac{1}{2}.
\end{align}

If $0<L_s<1$, then 
\begin{equation}\label{bNu}
    \Nu \lesssim L_s^{-\frac{1}{3}}\Ra^\frac{1}{3} + L_s^{-\frac{2}{3}} \Pr^{-\frac{1}{6}}\Ra^\frac{1}{2}+ L_s^{-\frac{2}{13}}\Ra^\frac{5}{13}+\Ra^\frac{5}{12}.
\end{equation}

\end{theorem}}
{We reserve the  discussion on physical implications of this result for the conclusions, in Section 4.} 
We first remark that, for $L_s\geq 1$, this result improves the upper bound in \eqref{DNN22}: in fact if $L_s\sim \Ra^{\alpha}$, with $\alpha\geq 0$, our new result yields 
\begin{equation}
    \Nu\lesssim 
    \begin{cases}
    \Ra^{\frac{5}{12}} & \mbox{if } \Pr\geq \Ra^{\frac 12-\alpha}\\\Pr^{-\frac 16}\Ra^{\frac{1}{2}-\frac{\alpha}{6}} & \mbox{if } \Pr\leq\Ra^{\frac 12-\alpha} \,.
    \end{cases}    
\end{equation}
In particular, when $\alpha=0$ we notice a crossover at $\Pr\sim \Ra^{\frac 12}$ between the $Ra^{\frac 12}$ and the $\Pr^{-\frac 16}\Ra^{\frac 12}$ scaling regime. This is reminiscent of the upper bound in \cite{choffrutNobiliOttoUpperBoundsOnNusseltNumberAtFinitePrandtlNumber} for no-slip boundary conditions. 
Our result also covers the case $0<L_s<1$ and, in this region, we can detect four scaling regimes depending on the magnitude of Prandtl number and $L_s<1$. 
The dominating terms in  \eqref{b1Nu} and \eqref{bNu} are summarized in Table \ref{table:overview_of_the_results}.
Observe that for $\Pr\rightarrow \infty$, the term $L_s^{-\frac 13}\Ra^{\frac 13}$ is dominating in the region $0<L_s<\Ra^{-\frac{2}{7}}$.
On one hand, when $\Pr\rightarrow \infty$, this seems to indicate the (expected) transition from Navier-slip to no-slip boundary conditions in the bounds. In order to contextualize this remark, we recall that for no-slip boundary conditions ($L_s=0$), the upper bound $\Nu\lesssim \Ra^{\frac 13}$ was proven when $\Pr=\infty$ (\cite{otto2011rayleigh}) and when $\Pr\gtrsim \Ra^{\frac 13}$ (\cite{choffrutNobiliOttoUpperBoundsOnNusseltNumberAtFinitePrandtlNumber}). On the other hand we stress that our bounding method breaks down in the limit $L_s\rightarrow 0$ and, consequently, all but the last scaling prefactors in \eqref{bNu} blow up.

\begin{table}
    \begin{center}
        \begin{tabular}{lccl}
            \multicolumn{2}{c}{Assumptions}  &\phantom{.}\hspace{25pt}\phantom{.}& \multicolumn{1}{c}{Bound}
            \\\hline
            & \textcolor{\colorBoth}{$\Ra^{-\frac{5}{24}}\leq L_s$} &&\textcolor{\colorBoth}{$\Nu \lesssim \Ra^{\frac{5}{12}}$}
            \\
            \multirow{2}{*}{$\Ra^\frac{11}{7}\leq \Pr$} & \textcolor{\colorSmall}{$\Ra^{-\frac{2}{7}}\leq L_s\leq \Ra^{-\frac{5}{24}}$} &&\textcolor{\colorSmall}{$\Nu \lesssim L_s^{-\frac{2}{13}}\Ra^{\frac{5}{13}}$}
            \\
            & \textcolor{\colorSmall}{$\Pr^{-\frac{1}{2}}\Ra^\frac{1}{2}\leq L_s\leq \Ra^{-\frac{2}{7}}$} && \textcolor{\colorSmall}{$\Nu \lesssim L_s^{-\frac{1}{3}}\Ra^{\frac{1}{3}}$}
            \\
            & \textcolor{\colorSmall}{$L_s\leq \Pr^{-\frac{1}{2}} \Ra ^{\frac{1}{2}}$} && \textcolor{\colorSmall}{$\Nu \lesssim L_s^{-\frac{2}{3}}\Pr^{-\frac{1}{6}}\Ra^\frac{1}{2}$}
            \\\hline
            & \textcolor{\colorBoth}{$\Ra^{-\frac{5}{24}}\leq L_s$} && \textcolor{\colorBoth}{$\Nu \lesssim \Ra^{\frac{5}{12}}$}
            \\
            $\Ra^\frac{4}{3}\leq \Pr\leq \Ra^{\frac{11}{7}}$ & \textcolor{\colorSmall}{$\Pr^{-\frac{13}{40}}\Ra^{\frac{9}{40}}\leq L_s\leq \Ra^{-\frac{5}{24}}$} &&\textcolor{\colorSmall}{$\Nu \lesssim L_s^{-\frac{2}{13}}\Ra^{\frac{5}{13}}$}
            \\
            & \textcolor{\colorSmall}{$L_s\leq \Pr^{-\frac{13}{40}} \Ra ^{\frac{9}{40}}$} && \textcolor{\colorSmall}{$\Nu \lesssim L_s^{-\frac{2}{3}}\Pr^{-\frac{1}{6}}\Ra^\frac{1}{2}$}
            \\\hline
            \multirow{2.2}{*}{$\Ra^\frac{1}{2}\leq \Pr\leq \Ra^\frac{4}{3}$} & \textcolor{\colorBoth}{$\Pr^{-\frac{1}{4}}\Ra^{\frac{1}{8}}\leq L_s$} && \textcolor{\colorBoth}{$\Nu \lesssim \Ra^{\frac{5}{12}}$}
            \\
            & \textcolor{\colorSmall}{$L_s\leq \Pr^{-\frac{1}{4}} \Ra ^{\frac{1}{8}}$} && \textcolor{\colorSmall}{$\Nu \lesssim L_s^{-\frac{2}{3}}\Pr^{-\frac{1}{6}}\Ra^\frac{1}{2}$}
            \\\hline
            & \textcolor{\colorBig}{$\Pr^{-1}\Ra^\frac{1}{2}\leq L_s$} && \textcolor{\colorBig}{$\Nu \lesssim \Ra^{\frac{5}{12}}$}
            \\
            $\Pr \leq \Ra^\frac{1}{2}$ & \textcolor{\colorBig}{$1\leq L_s \leq \Pr^{-1}\Ra^\frac{1}{2}$} && \textcolor{\colorBig}{$\Nu \lesssim L_s^{-\frac{1}{6}}\Pr^{-\frac{1}{6}}\Ra^{\frac{1}{2}}$}
            \\
            & \textcolor{\colorSmall}{$L_s\leq 1$} && \textcolor{\colorSmall}{$\Nu \lesssim L_s^{-\frac{2}{3}}\Pr^{-\frac{1}{6}}\Ra^\frac{1}{2}$}
        \end{tabular}
        \caption{Overview of the results in Theorem  \ref{Main-th}. The coloring corresponds the cases \textcolor{\colorSmall}{$L_s\leq 1$} and  \textcolor{\colorBig}{$1\leq L_s$}. In all the other (uncolored) cases, $L_s$ may be smaller or bigger than one.}
        \label{table:overview_of_the_results}
    \end{center}
\end{table}

Differently from the result in \cite{whiteheadDoeringUltimateState}, the proof of our theorem does not rely on the \textit{background field method} but rather exploits the regularity properties of the flow through a localization principle. In fact the Nusselt number can be localized in the vertical variable
\begin{equation}\label{loc_intro}
    \Nu=\frac{1}{\delta}\left\langle \int_0^{\delta} (u_2T-\partial_2 T)\, dx_2\right\rangle\leq \frac{1}{\delta}\left\langle \int_0^{\delta} u_2T\, dx_2\right\rangle+\frac{1}{\delta}.
\end{equation}
Here we used the boundary conditions for the temperature and the maximum principle $\sup_x|T(x,t)|\leq 1$ for all $t$.
Notice that this localization principle comes from the fact that the (long time and horizontal average of the) heat flux is the same for each $x_2\in (0,1)$ and this can be deduced by merely using the non-penetration boundary conditions $u_2=0$ at $x_2=0$, $x_2=1$ (see the proof of Lemma \ref{lemma_nusselt_localization} below).

The second crucial point in our proof is the following interpolation inequality 
\begin{equation}\label{inter}
    \frac{1}{\delta}\left\langle \int_{0}^{\delta} u_2T\,dx_2\right\rangle\leq \frac{1}{2}\left\langle \int_0^{1}|\partial_2 T|^2\, dx_2\right\rangle+C\delta^3\left\langle \int_0^{1}|\omega|^2\, dx_2\right\rangle^{\frac 12}\left\langle \int_0^{1}|\partial_1 \omega|^2\, dx_2\right\rangle^{\frac 12}
\end{equation}
{where $C$ is some positive constant,} which was first derived (in a slightly different form) in \cite{drivasNguyenNobiliBoundsOnHeatFluxForRayleighBenardConvectionBetweenNavierSlipFixedTemperatureBoundaries} and it is proved in Lemma \ref{lemma-interpol} below. We again observe that this bound holds only relying on the assumption $u_2=0$ at the boundaries $x_2=0, x_2=1$. 

In Lemma \ref{lemma-pressure} we prove the new pressure estimate
\begin{equation}\label{p-es}
\|p\|_{H^1}\leq C\left( \frac{1}{L_s}\|\partial_2 \bu\|_{L^2}+\frac{1}{\Pra}\|\omega\|_{L^2}\|\omega\|_{L^r}+\Ra \|T\|_{L^2}\right)\,,
\end{equation}\noeqref{p-es}
{where $L^2$ is the space of square integrable functions, while $H^1$ is the space of functions in $L^2$ with square integrable gradients. The space $L^r$ instead consists of functions whose $r$-th power is integrable. See definitions in \eqref{norm_definition}.} Notice that the improvement (compared to the pressure estimate in Proposition 2.7 in \cite{drivasNguyenNobiliBoundsOnHeatFluxForRayleighBenardConvectionBetweenNavierSlipFixedTemperatureBoundaries}) lies in the first term on the right-hand side, stemming from the new trace-type estimate 
\begin{equation}
    \bigg|\int_0^{\Gamma}(p\partial_1u_1|_{x_2=1}+p\partial_1u_1|_{x_2=0})\, dx_1\bigg|\leq 3\|p\|_{H^1}\|\partial_2 \bu\|_{L^2},
\end{equation}
used to control the boundary terms in the $H^1$ pressure identity \eqref{claim1}.
We observe that the new pressure estimate enables us to treat the case of small slip-length, i.e. $0< L_s\leq 1$. 
This small slip-length regime was not treatable in \cite{drivasNguyenNobiliBoundsOnHeatFluxForRayleighBenardConvectionBetweenNavierSlipFixedTemperatureBoundaries}{, see Remark \ref{remark_difference_pressure_bound}}.

\smallskip
\section*{Organization and notations}
The paper is divided in two sections: in Section 2 we prove all the a-priori estimates that we will need to prove the main theorem in Section 3. The crucial localization and interpolation lemmas are proven in Lemma \ref{lemma_nusselt_localization} and \ref{lemma-interpol} respectively. The improvement of the upper bounds on the Nusselt number stems from the new pressure estimates in Lemma \ref{lemma-pressure}. {In Section \ref{section_conclusion} we contextualize our result and give a physical interpretation of our bounds.}

{
Throughout the paper we will use the following Lebesgue and Sobolev norms
\begin{align}
    \|f\|_{L^p}^p = \int_{\Omega} |f|^p\ dx, \quad
    \|f\|_{W^{1,p}}^p = \int_{\Omega} |f|^p\ dx + \int_{\Omega} |\nabla f|^p\ dx, \quad
    \|f\|_{H^1} = \|f\|_{W^{1,2}}
    \label{norm_definition}
\end{align}
for any $1\leq p <\infty$.
}
 
\section{Identities and a-priori bounds}\label{sec:apriori}
In this section we derive a-priori bounds for the Rayleigh-B\'enard convection problem with Navier-slip boundary conditions. These will be used in the proof of Theorem \ref{Main-th} in the next section.

Recall that the temperature equation enjoys a maximum principle:
if $0\leq T_0(\mathbf{x})\leq 1$ then  
\begin{equation}\label{max-prin}
0\leq T(\mathbf{x},t)\leq 1 \mbox{ for all } \mathbf{x},t.
\end{equation}

{
The following lemma will allow us to localize the Nusselt number in a strip of height $\delta>0$, indicating the thermal boundary layer. This is the key ingredient of the \textit{direct method} and will later be used to bound the heat transfer.
}

\begin{lemma}[Localization of the Nusselt number]
\label{lemma_nusselt_localization}
The Nusselt number $\Nu$ defined in \eqref{Nusselt} is independent of $x_2$, that is 
\begin{equation}\label{loc}
\Nu=\langle u_2T-\partial_2 T\rangle\qquad \mbox{ for all } x_2\in [0,1].
\end{equation}
In particular, for any $\delta\in (0,1)$ it holds
\begin{equation}\label{localization_1}
    \Nu=\frac{1}{\delta}\left\langle \int_0^{\delta}(u_2T-\partial_2 T)\, dx_2\right\rangle 
\end{equation}
\end{lemma}
\begin{proof}
Taking the long-time and horizontal average of the equation for $T$ we have 
\begin{equation}
    0=\langle \partial_t T\rangle+\langle \bnabla \bcdot (\bu T-\bnabla T)\rangle= \partial_{2}\langle u_2T-\partial_2 T\rangle.
\end{equation}
Hence $\langle u_2T-\partial_2 T\rangle$ is constant in $x_2$, proving \eqref{loc}. The identity \eqref{localization_1} is a direct consequence of \eqref{loc}.
\end{proof}
Notice that from the boundary condition for $T$ at $x_2=0${, i.e. \eqref{temperature_bc},} and the maximum principle \eqref{max-prin} it follows
\begin{equation}
    -\int_0^\delta \partial_2 T \, dx_2= -T(x_1,\delta)+1\leq 1\,.
\end{equation}
As a consequence we obtain
\begin{align}\label{localization_2}
    \Nu \leq \frac{1}{\delta}\left\langle \int_0^{\delta}u_2T\, dx_2\right\rangle  + \frac{1}{\delta}\,.
\end{align}

Thanks to \eqref{loc}, we can now derive another useful identity {relating temperature gradients (naturally emerging in \eqref{omega_partial1T_term_estimate} and \eqref{interpolation-bound}) to the Nusselt number.}
\begin{lemma}[Representation of the Nusselt number]
The Nusselt number, defined in \eqref{Nusselt}, has the following alternative representation
\begin{equation}
    \Nu=\left\langle \int_0^{1}|\bnabla T|^2\, dx_2\right\rangle\,.\label{Nusselt_gradT}
\end{equation}
\end{lemma}
\begin{proof}
Testing {\eqref{eq:T},} the temperature equation with $T$ and integrating by parts we obtain 
\begin{equation}
    \frac{1}{2}\frac{d}{dt}\|T\|_{L^2}^2=-\|\bnabla T\|_{L^2}^2-\int_0^{\Gamma}\partial_2 T|_{x_2=0}\, dx_1,   
\end{equation}
where we used the incompressibility condition \eqref{incomp}, the non-penetration condition $\bu\cdot e_2=u_2=0$ at $x_2=\{0,1\}$ and the boundary conditions \eqref{temperature_bc} 
for $T$.
The statement follows from taking the long-time averages, observing that $\limsup_{t\rightarrow \infty}\int_0^t \frac{d}{ds}\frac{1}{\Gamma} \int_0^{\Gamma}\int_0^1|T|^2\, dx_2\,dx_1\, ds=0$ thanks to the maximum principle for $T$ \eqref{max-prin}, and using $\Nu=\langle u_2T-\partial_2 T\rangle|_{x_2=0}$ by \eqref{loc}. 
\end{proof}

{
The subsequent Lemma \ref{lemma_energy} provides a bound on the long time (and spatial) average of the velocity gradient, naturally arising from the interpolation estimate in  Lemma \ref{lemma-interpol}. Moreover this bound will be used to control the vorticity gradient in Lemma \ref{lemma_enstrophy}.}

\begin{lemma}[Energy]
\label{lemma_energy}
Let $0<L_s<\infty$ and suppose $\bu_0\in L^2$. Then there exists a constant $C=C(\Gamma)>0$ such that
\begin{equation}\label{boundedness}
    \|\bu(t)\|_{L^2}\leq \|\bu_0\|_{L^2} + C\max\{1,L_s\} \Ra
\end{equation}
for all times $t\in [0,\infty)$ and 
\begin{equation}\label{energy}
    \left\langle\int_0^1 |\bnabla \bu|^2\,dx_2\right\rangle\leq \Nu \Ra.
\end{equation}
{
If $L_s=\infty$ the bound \eqref{boundedness} simplifies to
\begin{align}
    \|\bu(t)\|_{L^2}\leq \|\bu_0\|_{L^2} + C\Ra
\end{align}
and \eqref{energy} remains valid.
}
\end{lemma}
\begin{proof}
{At first we assume $0<L_s<\infty$ and} test the velocity equation \eqref{navier-stokes} with $\bu$, integrate by parts, use the boundary conditions \eqref{nav-slip} for $u$ and the incompressibility condition $\eqref{incomp}$ to find
\begin{equation}
    \frac {1}{2\Pra}\frac{d}{dt}\int_{\Omega} |\bu|^2\, dx+ \int_{\Omega} |\bnabla \bu|^2\, dx+\frac{1}{L_s}\int_{0}^{\Gamma} (u_1^2|_{x_2=0}+u_1^2|_{x_2=1})\, dx_1=\Ra \int_{\Omega} Tu_2\, dx.\label{testing_nse}
\end{equation}
{
The fundamental theorem of calculus implies
\begin{align}
    u_1(\mathbf{x}) = u_1(x_1,0) + \int_0^{x_2} \partial_2 u_1(x_1,z)\, dz
\end{align}
and using Young's and Hölder's inequality
\begin{align}
    u_1^2(\mathbf{x}) &= u_1^2(x_1,0) + 2 u_1(x_1,0)\int_0^{x_2} \partial_2 u_1(x_1,z)\, dz + \left(\int_0^{x_2} \partial_2 u_1(x_1,z)\, dz\right)^2
    \\
    &\leq 2 u_1^2(x_1,0) + 2\left(\int_0^{x_2} |\partial_2 u_1(x_1,z)|\, dz\right)^2
    \\
    &\leq 2 u_1^2(x_1,0) + 2\int_0^{x_2} |\partial_2 u_1(x_1,z)|^2\, dz,
\end{align}
which after integration implies
\begin{align}
    \|u_1\|_{L^2}^2 \leq 2 \int_0^\Gamma u_1^2\vert_{x_2=0}\, dx_1 + 2 \|\partial_2 u_1\|_{L^2}^2.\label{poincare_type_estimate_u1}
\end{align}
By \eqref{nav-slip} one has $u_2=0$ on the boundaries and therefore the analogous estimate shows
\begin{align}
    \|u_2\|_{L^2}^2 \leq \|\partial_2 u_2\|_{L^2}^2.
    \label{poincare_type_estimate_u2}
\end{align}
The full vector norm of the velocity can now be split into the norms of its components and by \eqref{poincare_type_estimate_u1} and \eqref{poincare_type_estimate_u2} it holds
\begin{align}
    \|\bu\|_{L^2}^2 &= \|u_1\|_{L^2}^2 + \|u_2\|_{L^2}^2
    \\
    &\leq 2 \left(\|\partial_2 u_1\|_{L^2}^2 + \int_{0}^{\Gamma} (u_1^2|_{x_2=0}+u_1^2|_{x_2=1})\, dx_1 + \|\partial_2 u_2\|_{L^2}^2\right)
    \\
    &\leq 2 \left(\|\bnabla \bu\|_{L^2}^2 + \int_{0}^{\Gamma} (u_1^2|_{x_2=0}+u_1^2|_{x_2=1})\, dx_1\right),
    \label{energy_nav_slip_poincare_type}
\end{align}
}
which applied to \eqref{testing_nse} yields
\begin{align}
    \frac {1}{2\Pra}&\frac{d}{dt}\int_{\Omega} |\bu|^2\, dx+ \frac{1}{C}\min\{1,L_s^{-1}\}\|\bu\|_{L^2}^2\nonumber
    \\
    &\leq\Ra \int_{\Omega} Tu_2\, dx
    \,{\leq\Ra  \|T\|_{L^2}\|u_2\|_{L^2}}
    \,{\leq \frac{1}{4\epsilon}\|T\|_{L^2}^2\Ra^2 + \epsilon \|u_2\|_{L^2}^2}
    \leq \frac{1}{4\epsilon}\Gamma\Ra^2 + \epsilon \|u_2\|_{L^2}^2
    \label{energy_est_a}
\end{align}
{for some constant $C>0$}, where we used Hölder's inequality, Young's inequality and $\|T\|_{L^{\infty}}\leq 1$ because of the maximum principle \eqref{max-prin}. Setting $\epsilon = \frac{1}{2C}\min\{1,L_s^{-1}\}$ implies
\begin{align}
    \frac {1}{\Pra}\frac{d}{dt}\|\bu\|_{L^2}^2+ \frac{1}{C}\min\{1,L_s^{-1}\}\|\bu\|_{L^2}^2\leq C\max\{1,L_s\}\Gamma\Ra^2
    \label{energy_est_b}
\end{align}
and Grönwall's inequality now yields \eqref{boundedness}. Taking the long-time average of \eqref{testing_nse}, using \eqref{boundedness}, one has
\begin{equation}
    \left\langle\int_0^1 |\bnabla \bu|^2\, dx_2\right\rangle+\frac{1}{L_s}\left\langle (u_1^2|_{x_2=0} + u_1^2|_{x_2=1})\right\rangle=\Ra \left\langle \int_0^1 Tu_2\, dx_2\right\rangle.
    \label{energy_est_c}
\end{equation}
The claim follows by observing that, due to the boundary conditions for $T$, we have
\begin{equation}
    \Nu=\left\langle \int_0^1 Tu_2\, dx_2\right\rangle+1.
    \label{energy_est_d}
\end{equation}
{
If $L_s=\infty$, integrating the first component of \eqref{navier-stokes} in space yields
\begin{align}
    \frac{1}{\Pra}\frac{d}{dt} \int_{\Omega} u_1 \, dx = - \frac{1}{\Pra}\int_{\Omega} \bu\cdot \nabla u_1 \, dx + \int_{\Omega} \Delta u_1 \, dx - \int_{\Omega} \partial_1 p.
    \label{energy_free_average}
\end{align}
The first term on the right-hand side of \eqref{energy_free_average} vanishes after integration by parts due to the incompressibility condition \eqref{incomp} and the boundary conditions \eqref{nav-slip}. Similarly the second and third term on the right-hand side vanish due to Stokes' theorem, \eqref{nav-slip} and the periodicity in the horizontal direction, showing that the spatial average of $u_1$ is conserved. Therefore, due to the Galilean symmetry of the system we can assume $u_1$ to be average free. Consequently the Poincaré inequality (\cite{Evans22}, Section 5.8.1) implies that there exists a constant $C=C(\Gamma)>0$ such that
\begin{align}
    \|u_1\|_{L^2}\leq C \|\bnabla u_1\|_{L^2}
\end{align}
and combined with \eqref{poincare_type_estimate_u2} we find
\begin{align}
    \|\bu\|_{L^2} \leq C\|\bnabla \bu\|_{L^2},
    \label{enery_free_slip_poincare}
\end{align}
the analogous of \eqref{energy_nav_slip_poincare_type}. Using \eqref{enery_free_slip_poincare} instead of \eqref{energy_nav_slip_poincare_type} the arguments corresponding to \eqref{energy_est_a}-\eqref{energy_est_d} yield the bounds for $L_s=\infty$.
}
\end{proof}

In order to bound the second derivatives of $\bu$ we will exploit the equation for the vorticity $\omega=\partial_1 u_2-\partial_2 u_1$:
\begin{alignat}{2}
    \Pr^{-1}(\partial_t \omega + \bu\bcdot \bnabla \omega) - \rmDelta \omega &= \Ra \partial_1T &\quad &\mbox{ in } \Omega\qquad\
    \\
    \omega &= \tfrac{1}{L_s}u_1 &\quad &\mbox{ at } x_2=1
    \\
    \omega &= -\tfrac{1}{L_s}u_1 &\quad &\mbox{ at } x_2=0\,.
\end{alignat}
{Note that in the two-dimensional setting the vorticity is a scalar function and} {for any $0<L_s\leq \infty$}, we have
\begin{equation}\label{grad-iden}
 \|\bnabla \bu\|_{L^2}=\|\omega\|_{L^2}, \qquad \|\bnabla \bu\|_{L^p}\leq C\|\omega\|_{L^p}.
\end{equation}
While the identity in $L^2$ follows from a direct computation, the inequality in $L^p$ follows by elliptic regularity: in fact let $\psi$ be the the stream function for $\bu$, i.e. $\bu=\bnabla^{\perp}\psi=(-\partial_2 \psi, \partial_1 \psi)$. {Since $\partial_1 \psi = u_2 = 0$ at $x_2=1$ and $x_2 = 0$ and $\psi$ is only defined up to a constant we can choose it such that $\psi = 0$ on $x_2=0$. Therefore, using the fundamental theorem of calculus
\begin{align}
    \psi(x_1,1)= \psi(x_1,0) + \int_0^1 \partial_2 \psi(x_1,z) \ dz = -\int_0^1 u_1(x_1,z) \ dz 
    \label{psi_est_a}
\end{align}
and since $\psi$ is constant at $x_2=1$ averaging \eqref{psi_est_a} in $x_1$ yields $\psi(x_1,1)=-\frac{1}{\Gamma}\int_{\Omega} u_1\, dx$. Combining these observations with the direct computation $\rmDelta \psi = \bnabla^\perp \cdot \bnabla^\perp \psi = \bnabla^\perp \cdot \bu = \omega$ shows that $\psi$ is a solution of
\begin{alignat}{2}
        \rmDelta \psi&=\omega \qquad\ &&\mbox{ in } \Omega
        \\
        \psi &= -\frac{1}{\Gamma}\int_\Omega u_1 \, dx\quad &&\mbox{ at } x_2=1
        \\
        \psi &= 0\quad &&\mbox{ at } x_2=0
\end{alignat}
and $\tilde \psi = \psi + x_2 \int_\Omega u_1\, d\mathbf{x}$ solves}
\begin{alignat}{2}
        \rmDelta \tilde\psi&=\omega \qquad && \mbox{ in } \Omega
        \\
        \psi &= 0\quad &&\mbox{ at } x_2 \in \{0,1\}.
\end{alignat}
One has
\begin{align}
    \|\bnabla \bu\|_{L^p}^p
    &=\|\bnabla u_1\|_{L^p}^p+\|\bnabla u_2\|_{L^p}^p
    =\|-\bnabla \partial_2 \psi\|_{L^p}^p+\|\bnabla \partial_1 \psi\|_{L^p}^p
    \\
    &\leq \|\bnabla^2  \psi\|_{L^p}^p+\|\bnabla^2 \psi\|_{L^p}^p
    = \|\bnabla^2  \tilde \psi\|_{L^p}^p+\|\bnabla^2 \tilde\psi\|_{L^p}^p
    \leq C\|\omega\|_{L^p}^p,
\end{align}
where we used {the Calderon-Zygmund estimate (\cite{GT77}, Section 9.4)} in the last inequality.

{The following lemma provides a higher order version of Lemma \ref{lemma_energy}. The uniform in time bound will later be used to control the pressure terms arising due to the vorticity production on the boundary, while the long time average estimate will be used in the proof of the main theorem to estimate the thickness of the thermal boundary layer.}
{
\setnoclub[3]   
\begin{lemma}[Enstrophy]
\label{lemma_enstrophy}
Suppose $\bu_0\in W^{1,4}$ and {$0<L_s \leq \infty$}. Then there  exists a constant $C=C(\Gamma)>0$ such that for all $t>0$
\begin{equation}
\label{boundedness-omega}
   \|\omega(t)\|_{L^4}\leq C\max \left\{1, L_s^{-3}\right\}(\|\bu_0\|_{W^{1,4}}+\Ra)
\end{equation}
and 
\begin{equation}\label{enstrophy}
    \left\langle\int_0^1 |\bnabla \omega|^2\,dx_2\right\rangle\leq\frac{1}{L_s}\left|\left\langle p\partial_1 u_1|_{x_2=1}\right\rangle+\left\langle p\partial_1 u_1|_{x_2=0}\right\rangle\right|+\Nu\Ra^{\frac 32}.
\end{equation}
\end{lemma}
}
\begin{proof}
For the proof of \eqref{boundedness-omega} we refer the reader to 
\citep[Lemma 2.11]{drivasNguyenNobiliBoundsOnHeatFluxForRayleighBenardConvectionBetweenNavierSlipFixedTemperatureBoundaries}, where the case of $L_s\geq 1$ is covered. The same argument also yields the bound in case of $L_s<1$.
Testing the equation with $\omega$ and integrating by parts we obtain
\begin{multline}
    \frac{1}{2\Pra}\frac{d}{dt}\left(\|\omega\|_{L^2}^2+\frac{1}{L_s}\|u_1\|_{L^2(x_2=0)}^2+\frac{1}{L_s}\|u_1\|_{L^2(x_2=1)}^2\right) + \|\bnabla \omega\|_{L^2}^2\, dx\\
    +\frac{1}{L_s}\left[\int_0^{\Gamma}\partial_1u_1 p|_{x_2=0}\, dx_1+\int_0^{\Gamma}\partial_1u_1 p|_{x_2=1}dx_1\right]\, =\Ra\int_{\Omega}\omega \partial_1 T\, dx\,,
    \label{test-vorticity-equation}
\end{multline}
where we used $-\partial_2\omega=\rmDelta u_1=\frac{1}{\Pra}(\partial_tu_1+(\bu\bcdot \bnabla )u_1)-\partial_1 p$ and the fact that $\omega(\bu\bcdot \bnabla )u_1 = \pm \frac{1}{L_s} u_1^2 \partial_1 u_1$ at $x_2=\{0,1\}$  vanishes after integration (in $x_1$) due to periodicity. We integrate \eqref{test-vorticity-equation} in time and notice that, thanks to \eqref{grad-iden} and the trace estimate {{(\cite{Evans22}, Section 5.5)}}, the first bracket on the left-hand side is bounded by the $H^1$-norm of $\mathbf{u}$:
\begin{align}
    \|\omega\|_{L^2}^2 + \frac{1}{L_s} \|u_1\|_{L^2(x_2=0)}^2+\frac{1}{L_s}\|u_1\|_{L^2(x_2=1)}^2 \leq C(L_s) \|\mathbf{u}\|_{H^1}^2.
\end{align}
Note also that due to Hölder's inequality, \eqref{boundedness-omega} additionally yields
\begin{align}
    \|\omega(t)\|_{L^2}\leq C\max \left\{1, L_s^{-3}\right\}(\|\bu_0\|_{W^{1,4}}+\Ra),
\end{align}
which, combined with \eqref{boundedness} and \eqref{grad-iden}, implies that $\|\mathbf{u}(t)\|_{H^1}$ is universally bounded in time. 
Therefore claim \eqref{enstrophy} follows from taking the space and long-time average of \eqref{test-vorticity-equation}, using the fact that the long-time average of the first term in \eqref{test-vorticity-equation} vanishes due to the argument above, and observing
\begin{equation}
    \label{omega_partial1T_term_estimate}
    \left\langle\int_{\Omega}\omega \partial_1 T\, dx_2\right\rangle \leq \left\langle\int_{\Omega}|\omega|^2  dx_2\right\rangle^{\frac 12} \left\langle\int_{\Omega}|\bnabla T|^2  dx_2\right\rangle^{\frac 12} \leq (\Nu\Ra)^{\frac 12}\Nu^{\frac 12},
\end{equation}
where we used \eqref{grad-iden}, \eqref{energy} and \eqref{Nusselt_gradT}.
\end{proof}
Notice that the pressure term appears at the boundary in \eqref{enstrophy} and, for this reason, we need to control its $H^1-$norm. {The following lemma will provide control for this term.}
Taking the divergence of \eqref{navier-stokes}, it is easy to see that the pressure $p$ satisfies
\begin{alignat}{2}
        \rmDelta p &= -\tfrac{1}{\Pra}\bnabla \bu^{\textup{T}}\bcolon\bnabla \bu+\Ra \partial_2 T  &\quad&\mbox{ in } \Omega
        \\
        \partial_2 p &= \tfrac{1}{L_s}\partial_1u_1\phantom{- \Ra}&\quad&\mbox{ at } x_2=1
        \label{pressure_pde}
        \\
        -\partial_2 p &= \tfrac{1}{L_s}\partial_1u_1- \Ra &\quad&\mbox{ at } x_2=0,
\end{alignat}
where the boundary conditions are derived by tracing the second component of \eqref{navier-stokes} on the boundary.

{
\setnoclub[3]
\begin{lemma}[Pressure bound]\label{lemma-pressure}
Let $r>2$. Then there exists a constant $C=C(r,\Gamma)>0$ such that
\begin{equation}\label{pressure}
\|p\|_{H^1}\leq C\left( \frac{1}{L_s}\|\partial_2 \bu\|_{L^2}+\frac{1}{\Pra}\|\omega\|_{L^2}\|\omega\|_{L^r}+\Ra \|T\|_{L^2}\right)\,.
\end{equation}
\end{lemma}
}
\begin{proof}
The proof is a consequence of the following two claims:
\begin{align}
    \|\bnabla p\|_{L^2}^2=&\frac{1}{L_s}\int_0^{\Gamma}(p\partial_1u_1|_{x_2=1}+p\partial_1u_1|_{x_2=0})\, dx_1+\frac{1}{\Pra}\int_{\Omega}p\bnabla \bu^{\textup{T}}\bcolon\bnabla \bu\, dx\nonumber
    \\
    &+\Ra\int_{\Omega}\partial_2 p T\, dx,\label{claim1}
    \\
    \bigg|\int_0^{\Gamma}&(p\partial_1u_1|_{x_2=1}+p\partial_1u_1|_{x_2=0})\, dx_1\bigg|\leq 3\|p\|_{H^1}\|\partial_2 \bu\|_{L^2}.
    \label{claim2}
\end{align}

In fact, applying {the Poincaré inequality} \footnote{$p$ can be assumed to have zero mean, since $p-\langle p\rangle$ satisfies the equations \eqref{pressure_pde}} and combining \eqref{claim1} and \eqref{claim2}, we obtain
\begin{align}
    \| p\|_{H^1}^2 &\lesssim\frac{1}{L_s}\int_0^{\Gamma}(p\partial_1u_1|_{x_2=1}+p\partial_1u_1|_{x_2=0})\, dx_1+\frac{1}{\Pra}\int_{\Omega}p\bnabla \bu^T\bcolon\bnabla \bu\, dx\nonumber
    \\
    &\quad\ +\Ra\int_{\Omega}\partial_2 p T\, dx
    \label{p_est_1}
    \\
    &\lesssim L_s^{-1}\|p\|_{H^1}\|\partial_2 \bu\|_{L^2}+\Pr^{-1}\|p\|_{L^q}\|\bnabla \bu\|_{L^2}\|\bnabla \bu\|_{L^r}+\Ra\|p\|_{H^1}\|T\|_{L^2}
    \label{p_est_2}
    \\
    &\lesssim L_s^{-1}\|p\|_{H^1}\|\partial_2 \bu\|_{L^2}+\Pr^{-1}\|p\|_{H^1}\|\bnabla \bu\|_{L^2}\|\bnabla \bu\|_{L^r}+\Ra\|p\|_{H^1}\|T\|_{L^2},
    \label{p_est_3}
\end{align}\noeqref{p_est_1}%
where {$f\lesssim g$ indicates that there exists $C>0$ such that $f\leq C g$ and} $q$ and $r$ are related by $\frac 1p+\frac 1r=\frac 12$. In {\eqref{p_est_2} we used the trace estimate and in \eqref{p_est_3} the fact that $\|p\|_{L^q}\leq C \|p\|_{H^1}$ }for any $q\in (2,\infty)$ by Sobolev embedding in bounded domains.
It is left to prove the claims.

\textit{Argument for \eqref{claim1}.}
Integrating by parts and using \eqref{pressure_pde}
\begin{align}
    \|\bnabla p\|_{L^2}^2 &= \int_0^\Gamma p\partial_2 p \vert_{x_2=1} \, dx_1 -  \int_0^\Gamma p\partial_2 p \vert_{x_2=0} \, dx_1 - \int_{\Omega} p\rmDelta p \, dx
    \\
    &= \frac{1}{L_s} \int_0^\Gamma p (\partial_1 u_1 \vert_{x_2=1} + \partial_1 u_1\vert_{x_2=0}) \, dx_1 - \Ra \int_0^\Gamma p \vert_{x_2=0} \, dx_1\nonumber
    \\
    &\quad +\frac{1}{\Pra} \int_{\Omega} p\bnabla \bu^{\textup{T}}\bcolon\bnabla \bu \, dx - \Ra \int_{\Omega}\partial_2 T p \, dx
    \\
    &= \frac{1}{L_s} \int_0^\Gamma p (\partial_1 u_1 \vert_{x_2=1} + \partial_1 u_1\vert_{x_2=0}) \, dx_1 +\frac{1}{\Pra} \int_{\Omega} p\bnabla \bu^{\textup{T}}\bcolon\bnabla \bu \, dx + \Ra \int_{\Omega} T \partial_2p \, dx,
\end{align}
where in the last identity we used $ - \Ra \int_{\Omega}\partial_2 T p- \Ra \int_0^\Gamma p \vert_{x_2=0} \, dx_1=  \Ra \int_{\Omega} T \partial_2 p \, dx$ thanks to the boundary conditions for $T$.

\textit{Argument for \eqref{claim2}.}
Since $\bu$ is divergence free we have 
\begin{equation}
    0=\int_{\Omega}p(-1+2x_2)\partial_2(\bnabla \bcdot \bu)\, dx=\int_{\Omega}p(-1+2x_2)\bnabla \bcdot (\partial_2\bu)\, dx,
\end{equation}
and integration by parts yields
\begin{align}
    \int_{\Omega}p(-1+2x_2)\bnabla \bcdot (\partial_2\bu)\, dx
    &=-\int_{\Omega}\bnabla p\bcdot (-1+2x_2)\partial_2 \bu\, dx - 2\int_{\Omega}p \partial_2 u_2\, dx\nonumber
    \\
    &\quad -\left(\int_0^{\Gamma}p\partial_1 u_1|_{x_2=1}+p\partial_1 u_1|_{x_2=0}\right)\, dx_1,
\end{align}
where we used $\partial_1 u_1=-\partial_2 u_2$ by incompressibility.
Combining the two identities we obtain
\begin{align}
    \left|\int_0^{\Gamma}(p\partial_1 u_1|_{x_2=1}+p\partial_1 u_1|_{x_2=0})\, dx_1\right|
    &\leq  2\|p\|_{L^{2}}\|\partial_2 u_2\|_{L^2}+\|\partial_2\bu\|_{L^2}\|\bnabla p\|_{L^2}
    \\
    &\leq 3\|p\|_{H^1}\|\partial_2\bu\|_{L^2}\,.
\end{align}
\end{proof}

{
\begin{remark}
    \label{remark_difference_pressure_bound}
    The pressure bound in Proposition 2.7 of \cite{drivasNguyenNobiliBoundsOnHeatFluxForRayleighBenardConvectionBetweenNavierSlipFixedTemperatureBoundaries} is given by
    \begin{align}
        \|p\|_{H^1}\leq C\left( \frac{1}{L_s}\|\partial_1 \omega\|_{L^2}+\frac{1}{\Pra}\|\omega\|_{L^2}\|\omega\|_{L^r}+\Ra \|T\|_{L^2}\right)\,.
    \end{align}
    In the subsequent analysis the authors bound the term $\frac{1}{L_s}\|\partial_1 \omega\|_{L^2}$ from above with $\|\bnabla \omega\|_{L^2}$ 
    imposing the conditions that $L_s\geq 1$. In contrast, using the refined estimate \eqref{pressure}, we are able to treat any slip length $L_s>0$, improving the scaling of $\Nu$ with respect to $\Ra$.
\end{remark}
}

We conclude this section with bounds on second derivatives of the velocity field. First we relate the $L^2$-norm of $\bnabla ^2 \bu$ to the $L^2$-norm of $\bnabla \omega$.
\begin{lemma}\label{lemma_hessian_iden}
\begin{equation}\label{hessian-iden}
\|\bnabla ^2 \bu\|_{L^2}\leq\|\bnabla \omega\|_{L^2}\,.
\end{equation}
\end{lemma}

\begin{proof}
First we show that $\|\bnabla ^2 \bu\|_{L^2}\leq\|\Delta \bu\|_{L^2}$:
integrating by parts twice yields
\begin{align}
    \|\bnabla^2 \bu\|_{L^2(\Omega)}^2&=\int_{\Omega}\partial_i\partial_j u_k \partial_i\partial_j u_k\, dx
    \\
    &=\int_{\Omega}\partial_i^2 u_k\partial_j^2 u_k \, dx-\int_{\partial\Omega}\partial_i^2 u_k\partial_j u_k n_j \, dS+\int_{\partial\Omega}\partial_i\partial_j u_k\partial_j u_k n_i \, dS 
    \\
    &=\|\rmDelta \bu\|_{L^2(\Omega)}^2-\int_0^\Gamma\partial_1^2 u_k\partial_2 u_k n_2 \, dx_1+\int_0^\Gamma\partial_2\partial_1 u_k\partial_1 u_k n_2 \, dx_1\,,
    \label{hessian_part_int_twice}
\end{align}
where we used periodicity in the horizontal direction and the fact that the terms with $i=j$ cancel. Note that,  due to \eqref{nav-slip}, the boundary terms have a sign:
\begin{align}
    -\int_0^\Gamma&\partial_1^2 u_k\partial_2 u_k n_2 \, dx_1+\int_0^\Gamma\partial_2\partial_1 u_k\partial_1 u_k n_2 \, dx_1\nonumber
    \\
    &= -\int_0^\Gamma\partial_1^2 u_1\partial_2 u_1 n_2 \, dx_1+\int_0^\Gamma\partial_2\partial_1 u_1\partial_1 u_1 n_2 \, dx_1
    \\
    &= \frac{1}{L_s}\int_0^\Gamma\partial_1^2 u_1 u_1 \, dx_1-\int_0^\Gamma(\partial_1 u_1)^2\, dx_1 = -\frac{2}{L_s}\int_0^\Gamma(\partial_1 u_1)^2\, dx_1 \leq 0\,,
    \label{hessian_boundary_terms}
\end{align}
where in the last identity we used the periodicity in the horizontal direction. This proves the first claim.

Now, a direct computation yields $\rmDelta \bu=\bnabla^{\perp}\omega$, and from this follows $\|\rmDelta \bu\|_{L^2(\Omega)}=\|\bnabla\omega\|_{L^2(\Omega)}$.
We conclude that
\begin{equation}
    \|\bnabla^2 \bu\|_{L^2(\Omega)} \leq \|\rmDelta \bu\|_{L^2(\Omega)}=\|\bnabla\omega\|_{L^2(\Omega)}
\end{equation}
yielding \eqref{hessian-iden}.
\end{proof}

Next, we relate $\|\bnabla^2 \bu\|_{L^2}$ to $\Nu,\Ra$ and $L_s$, via an upper bound for $\|\bnabla \omega\|_{L^2}$. This is done by combining {\eqref{enstrophy}}, the pressure bound \eqref{pressure} and the boundary integral estimate \eqref{claim2}. {The resulting bound on the long time average of the velocity hessian together with the corresponding bound on the velocity gradient of Lemma \ref{energy} are key ingredients in order to estimate the boundary layer thickness in the proof of the main theorem.}

\begin{lemma}[Hessian bound]
Let $L_s>0$ and $\bu_0\in W^{1,4}$. Then there exists a constant $C=C(\Gamma)>0$ such that
\begin{align}\label{hessian-ta}
    \quad\big\langle\|\bnabla^2 \bu\|_{L^2}^2 \big\rangle \leq C \Big( L_s^{-2}+\frac{\max\{1,L_s^{-3}\}(\|\bu_0\|_{W^{1,4}}+\Ra)}{\Pra L_s}+L_s^{-1}\Nu^{-\frac 12}\Ra^{\frac 12}+\Ra^{\frac 12}\Big)\Nu\Ra.\quad
\end{align}
\end{lemma}
\begin{proof}
From \eqref{hessian-iden} and \eqref{enstrophy} we have 
\begin{equation}
    \left\langle \int_0^1 |\bnabla^2 \bu|^2\, dx\right\rangle\leq L_s^{-1}\left(\left\langle p\partial_1 u_1|_{x_2=1}\right\rangle+\left\langle p\partial_1 u_1|_{x_2=0}\right\rangle\right)+\Nu\Ra^{\frac 32}.
\end{equation}
Using \eqref{claim2} and \eqref{pressure} we obtain
\begin{align}
    \left\langle \int_0^1 |\bnabla^2 \bu|^2\, dx\right\rangle &\leq L_s^{-1}\langle\|p\|_{H^1}\|\partial_2 \bu\|_{L^2}\rangle+ \Nu\Ra^{\frac 32}
    \\
    &\leq L_s^{-2}\langle\|\partial_2 \bu\|_{L^2}^2\rangle+L_s^{-1}\Pr^{-1}\langle\|\omega\|_{L^2}\|\omega\|_{L^4}\|\partial_2 \bu\|_{L^2}\rangle\nonumber
    \\
    &\quad\ +L_s^{-1}\Ra \langle\|T\|_{L^2}\|\partial_2 \bu\|_{L^2}\rangle+\Nu\Ra^{\frac 32}.
    \label{hessian_bound_proof_est_1}
\end{align}
{
For the $L^4$-norm of the vorticity the pointwise (in time) bound \eqref{boundedness-omega} states
\begin{align}
   \|\omega(t)\|_{L^4}\leq C\max \left\{1, L_s^{-3}\right\}(\|\bu_0\|_{W^{1,4}}+\Ra)
\end{align}
for some constant $C>0$ depending only on $\Gamma$. Therefore, additionally using the identity \eqref{grad-iden}, \eqref{hessian_bound_proof_est_1} can be estimated as follows
}
\begin{align}
    \left\langle \int_0^1 |\bnabla^2 \bu|^2\, dx_2\right\rangle &\leq L_s^{-2}\langle\|\partial_2 \bu\|_{L^2}^2\rangle +\frac{C\max\{1,L_s^{-3}\}}{\Pra L_s}\langle\|\bnabla \bu\|_{L^2}(\|\bu_0\|_{W^{1,4}}+\Ra)\|\partial_2 \bu\|_{L^2}\rangle \nonumber
    \\
    &\quad\ +L_s^{-1}\Ra \langle\|T\|_{L^2}\|\partial_2 \bu\|_{L^2}\rangle+\Nu\Ra^{\frac 32}.
\end{align}
Finally using $\langle\|\partial_2 \bu\|_{L^2}^2\rangle\leq\langle\|\bnabla  \bu\|_{L^2}^2\rangle $, the upper bound in \eqref{energy} and the maximum principle for the temperature \eqref{max-prin} we deduce
\begin{align}
    \left\langle \int_0^1 |\bnabla^2 \bu|^2\, dx_2\right\rangle
    &\leq L_s^{-2}\Nu\Ra +\frac{C\max\{1,L_s^{-3}\}}{\Pra L_s}(\|\bu_0\|_{W^{1,4}}+\Ra)\Nu\Ra \nonumber
    \\
    &\quad\ +L_s^{-1}\Gamma^{\frac 12}\Nu^{\frac 12}\Ra^{\frac 32}+\Nu\Ra^{\frac 32}.
\end{align}
\end{proof}
\section{Proof of Theorem \ref{Main-th} }\label{sec:proof}

A crucial ingredient of the proof of Theorem \ref{Main-th} is the following interpolation bound, relating the integral of the product $u_2T$ with $\|\bnabla T\|_{L^2}$, $\|\bnabla \bu\|_{L^2}$ and $\|\bnabla^2 \bu\|_{L^2}$. Notice that, in turn, these quantities are estimated in terms of Nusselt, Rayleigh and Prandtl numbers in \eqref{Nusselt_gradT}, \eqref{energy} and \eqref{hessian-ta}.
\begin{lemma}\label{lemma-interpol}
The interpolation bound 
\begin{equation}\label{interpolation-bound}
    \frac{1}{\delta}\left\langle \int_{0}^{\delta} u_2T\,dx_2\right\rangle\leq \frac{1}{2}\left\langle \int_0^1|\partial_2 T|^2\, dx_2\right\rangle+C\delta^3\left\langle \int_0^1|\bnabla \bu|^2\, dx_2\right\rangle^{\frac 12}\left\langle \int_0^1|\bnabla^2 \bu|^2\, dx_2\right\rangle^{\frac 12}
\end{equation}
holds for some constant $C>0$.
\end{lemma}
\begin{proof}
First notice that, due to incompressibility $\partial_2u_2=-\partial_1u_1$ and horizontal periodicity
\begin{equation}
    \partial_2 \frac{1}{\Gamma}\int_0^\Gamma u_2(x_1,x_2)\, dx_1=-\frac{1}{\Gamma}\int_0^\Gamma \partial_1 u_1(x_1,x_2)\, dx_1 = 0,
\end{equation}
implying
\begin{align}
    \frac{1}{\Gamma}\int_0^\Gamma u_2(x_1,x_2)\, dx_1=0
    \label{u2_horizontal_average_free}
\end{align}
thanks to the boundary conditions $u_2=0$ at $x_2=\{0,1\}$.
Let $\theta(x_1,x_2):=T(x_1,x_2)-1$, then $\theta(x_1, x_2)=0$ at $x_2=0$ and by \eqref{u2_horizontal_average_free}
\begin{equation}
    \frac{1}{\delta}\left\langle \int_{0}^{\delta} u_2T\,dx_2\right\rangle=\frac{1}{\delta}\left\langle \int_{0}^{\delta} u_2\theta\,dx_2\right\rangle\,.
\end{equation}
By the fundamental theorem of calculus and the homogeneous Dirichlet boundary conditions for $\theta$ and $u_2$ we have 
\begin{align}
    |\theta(x_1,x_2)|&\leq x_2^{\frac 12}\|\partial_2\theta\|_{L^2(0,1)}\\
    |u_2(x_1,x_2)|&\leq x_2\|\partial_2u_2\|_{L^{\infty}(0,1)}\,.
\end{align}
Furthermore, since $\int_0^1\partial_2u_2(x_1,x_2)\, dx_2=0$, there exists $\xi=\xi(x_1)\in (0,1)$ such that $\partial_2u_2(x_1,\xi)=0$.
Then
\begin{equation}
    |\partial_2u_2(x_1,x_2)|^2=\left|\int_{\xi}^{x_2}\partial_2(\partial_2 u_2(x_1,z))^2\, dz\right|=\left|2\int_{\xi}^{x_2}\partial_2 u_2\partial_2^2 u_2 \,dz\right|,
\end{equation}
which implies
\begin{equation}
    |u_2(x_1,x_2)|\leq x_2\|\partial_2u_2(x_1,\cdot)\|_{L^2(0,1)}^{\frac 12}\|\partial_2^2u_2(x_1,\cdot)\|_{L^2(0,1)}^{\frac 12}\,.
\end{equation}
Combining these estimates we obtain 
\begin{align}
    \frac{1}{\delta}&\left\langle \int_0^{\delta}u_2\theta\, dx_2\right\rangle \nonumber
    \\
    &\qquad\leq \frac{1}{\delta}\left\langle \int_0^{\delta}x_2^{\frac 32} \|\partial_2u_2(x_1,\cdot)\|_{L^2(0,1)}\|\partial_2^2u_2(x_1,\cdot)\|_{L^2(0,1)}\|\partial_2\theta(x_1,\cdot)\|_{L^2(0,1)}^2\, dx_2\right\rangle
    \\
    &\qquad\leq \delta^{\frac 32}\left\langle \|\partial_2u_2(x_1,\cdot)\|_{L^2(0,1)}\|\partial_2^2u_2(x_1,\cdot)\|_{L^2(0,1)}\right\rangle^{\frac 12}\left\langle \|\partial_2\theta(x_1,\cdot)\|_{L^2(0,1)}^2 \right\rangle^{\frac 12}
    \\
    &\qquad\leq C\delta^{3}\left\langle \|\partial_2u_2(x_1,\cdot)\|_{L^2(0,1)}^2\right\rangle^{\frac 12}\left\langle \|\partial_2^2u_2(x_1,\cdot)\|_{L^2(0,1)}^2\right\rangle^{\frac 12} +\frac 12\left\langle \|\partial_2\theta(x_1,\cdot)\|_{L^2(0,1)}^2 \right\rangle.
\end{align}
\end{proof}
{\subsection{Case \texorpdfstring{$L_s=\infty$}{of free slip}}
In this section we prove the upper bound \eqref{standard} when $L_s=\infty$ in \eqref{nav-slip}, re-deriving the seminal result of \cite{whiteheadDoeringUltimateState} with a different technique. While the proof in \cite{whiteheadDoeringUltimateState} is a sophisticated application of the \textit{background field method}, our new proof for the $\Ra^{\frac{5}{12}}$-scaling is a pure PDE argument based on the combination of the localization principle together with the interpolation bound \eqref{inter}.

We first notice that, setting $L_s=\infty$ in \eqref{energy}, \eqref{grad-iden} and \eqref{enstrophy} we obtain 
\begin{alignat}{2}
    \left\langle \int_0^{1}|\omega|^2\, dx_2\right\rangle&=\left\langle \int_0^{1}|\bnabla \bu|^2\, dx_2\right\rangle&&\leq \Nu\Ra,
    \\
    \left\langle \int_0^{1}|\partial_1\omega|^2\, dx_2\right\rangle&\leq\left\langle \int_0^{1}|\bnabla \omega|^2\, dx_2\right\rangle &&\leq \Nu\Ra^{\frac 32}.
\end{alignat}
Combining these upper bounds and the identity \eqref{Nusselt_gradT} in the localization estimate \eqref{loc_intro}, we find
\begin{equation}
    \Nu\leq \frac{1}{2}\Nu +C\delta^3(\Nu\Ra)^{\frac 12}(\Ra^{\frac 32}\Nu)^{\frac 12}+\frac{1}{\delta}
\end{equation}
which yields 
\begin{equation}
    \frac 12\Nu\leq C\delta^3 \Nu\Ra^{\frac 54}+\frac{1}{\delta}\,.
\end{equation}
Optimizing in $\delta$
\begin{equation}
    \delta\sim\frac{1}{\Ra^{\frac{5}{16}}\Nu^{\frac{1}{4}}},
\end{equation}
we deduce
\begin{align}
    \Nu\lesssim\Ra^{\frac{5}{12}}.    
\end{align}}

\subsection{Case \texorpdfstring{$0<L_s<\infty$}{of finite slip}}

Using the localization of the Nusselt number \eqref{localization_2} and \eqref{interpolation-bound} we have
\begin{align}
    \Nu&\leq \frac{1}{\delta}\left\langle \int_{0}^{\delta} u_2T\,dx_2\right\rangle+\frac{1}{\delta}
    \\
    &\leq \frac{1}{2}\left\langle \int_0^{1}|\bnabla T|^2\, dx_2\right\rangle+C\delta^3\left\langle \int_0^{1}|\bnabla \bu|^2\, dx_2\right\rangle^{\frac 12}\left\langle \int_0^{1}|\bnabla^2 \bu|^2\, dx_2\right\rangle^{\frac 12}+\frac{1}{\delta}
\end{align}
yielding 
\begin{equation}
    \frac 12\Nu\leq C\delta^3\left\langle \int_0^{\delta}|\bnabla \bu|^2\, dx_2\right\rangle^{\frac 12}\left\langle \int_0^{1}|\bnabla^2 \bu|^2\, dx_2\right\rangle^{\frac 12}+\frac{1}{\delta},
\end{equation}
where we used \eqref{Nusselt_gradT}.
Finally we insert the bounds \eqref{energy} and \eqref{hessian-ta} in the last inequality and, if $\Ra$ is large enough such that $\|\bu_0\|_{W^{1,4}}\lesssim \Ra$ then, up to redefining constants, we find
\begin{equation}\label{Up-bound-standard}
    \Nu\leq C\delta^3\Big(L_s^{-1}\Nu\Ra+\max\{1,L_s^{-\frac{3}{2}}\}L_s^{-\frac{1}{2}}\Pr^{-\frac{1}{2}}\Nu\Ra^{\frac 32}+L_s^{-\frac{1}{2}}\Nu^{\frac 34}\Ra^{\frac 54}+\Nu\Ra^{\frac 54}\Big)+\frac{2}{\delta}.
\end{equation}
We first cover the case $L_s\geq 1$. Then, as $L_s,\Nu\geq 1$, the first and third terms on the right-hand side of \eqref{Up-bound-standard} are dominated by $\Nu\Ra^{\frac 54}$, hence
\begin{align}
    \Nu\leq C\delta^{3}\Nu\left(L_s^{-\frac{1}{2}}\Pr^{-\frac{1}{2}}\Ra^\frac{3}{2}+\Ra^\frac{5}{4}\right)+ 2\delta^{-1}
\end{align}
and optimizing by setting $\delta = \Nu^{-\frac{1}{4}}\Big(L_s^{-\frac{1}{2}}\Pr^{-\frac{1}{2}}\Ra^\frac{3}{2}+\Ra^\frac{5}{4}\Big)^{-\frac{1}{4}}$ implies
\begin{align}
    \Nu \lesssim L_s^{-\frac{1}{6}}\Pr^{-\frac{1}{6}}\Ra^\frac{1}{2}+\Ra^\frac{5}{12}.
\end{align}
If instead $L_s<1$, then \eqref{Up-bound-standard} is given by
\begin{align}
    \Nu\leq C\delta^3\Big(L_s^{-1}\Nu\Ra+L_s^{-2}\Pr^{-\frac{1}{2}}\Nu\Ra^{\frac 32}+ L_s^{-\frac{1}{2}}\Nu^{\frac 34}\Ra^{\frac 54}+\Nu\Ra^{\frac 54}\Big)+2 \delta^{-1}.    
\end{align}
Optimizing by setting 
\begin{equation}
    \delta=
    \begin{cases}
        \Nu^{-\frac{1}{4}}\left(L_s^{-1}\Ra+L_s^{-2}\Pra^{-\frac12} \Ra^{\frac 32}+\Ra^{\frac 54}\right)^{-\frac{1}{4}} & \mbox{ if } L_s^{-\frac{1}{2}}\Ra^{-\frac{1}{4}}+L_s^{-\frac{3}{2}}\Pra^{-\frac 12}\Ra^{\frac 14}+L_s^\frac{1}{2} \geq \Nu^{-\frac{1}{4}}\\
        L_s^{\frac{1}{8}}\Nu^{-\frac{3}{16}}\Ra^{-\frac{5}{16}} & \mbox{ if } L_s^{-\frac{1}{2}}\Ra^{-\frac{1}{4}}+L_s^{-\frac{3}{2}}\Pra^{-\frac 12}\Ra^{\frac 14}+L_s^\frac{1}{2} \leq \Nu^{-\frac{1}{4}}
    \end{cases}
\end{equation}
yields
\begin{equation}
    \Nu \lesssim L_s^{-\frac{1}{3}}\Ra^\frac{1}{3} + L_s^{-\frac{2}{3}} \Pr^{-\frac{1}{6}}\Ra^\frac{1}{2}+ L_s^{-\frac{2}{13}}\Ra^\frac{5}{13}+\Ra^\frac{5}{12}.
\end{equation}
{\hspace*{\fill}\qedsymbol}
{
\section{Conclusion}\label{section_conclusion}

Our work is motivated by recent investigations in \cite{whiteheadDoeringUltimateState}, \cite{drivasNguyenNobiliBoundsOnHeatFluxForRayleighBenardConvectionBetweenNavierSlipFixedTemperatureBoundaries} indicating that variations in the boundary conditions for the velocity affect heat transport properties in Rayleigh-B\'enard convection. The Navier-slip boundary conditions are prescribed in the situation in which some slip occurs on the surface and, at the same time, stress is exerted on the fluid. In particular, the use of these boundary conditions is justified when the scale of interest goes down to micron or below, when the no-slip boundary conditions cease to be valid (see \cite{choi2006large}). Mathematically these boundary conditions are referred to as "interpolation" boundary conditions since they represent a situation in between the physically relevant, but very difficult to treat, no-slip boundary conditions and the easier, but unphysical free-slip boundary conditions. When free-slip boundary conditions are considered, the key tool in order to control the growth of the vertical velocity $u_2$ near the boundaries is the enstrophy balance \cite{whiteheadDoeringUltimateState, DoeringWhitehead12rigid, WangWhitehead13}. When Navier-slip boundary condition are considered instead, the control of enstrophy production becomes difficult since the pressure term and the vertical derivative of $u_2$, $\partial_2u_2=-\partial_1 u_1$, appear at the boundary, see \eqref{enstrophy}. This problem was already tackled in \cite{drivasNguyenNobiliBoundsOnHeatFluxForRayleighBenardConvectionBetweenNavierSlipFixedTemperatureBoundaries} and, in this paper, we improve this control by refining the trace estimates, see \eqref{claim2}. In \cite{bleitnerNobili24} the analysis in \cite{drivasNguyenNobiliBoundsOnHeatFluxForRayleighBenardConvectionBetweenNavierSlipFixedTemperatureBoundaries} was extended to \textit{rough} boundary conditions, capturing the dependency of the scaling from the spatially varying friction coefficient and curvature. To our knowledge, no other rigorous results are available for the Nusselt number in this set up. 

In the same setting considered in the present paper, direct numerical simulations (DNS) were performed in \cite{huang2022effect} to study the scaling of the Nusselt number in the convective roll state and in the zonal flow (turbulent state). The authors found that, when $\frac{L_s}{\lambda_0}\lesssim 10$ (here the slip-length is normalized by $\lambda_0$, the thermal boundary layer thickness for the no-slip plates) then $\Nu\sim \Ra^{0.31}$ in the convection roll state. In particular they observe that, in the convection roll state, for a fixed $\Ra$, the heat transfer $\Nu$ increases with increasing $\frac{L_s}{\lambda_0}$. On the other hand, they observe that \textit{in the zonal flow the heat transport $\Nu$ decreases with increasing $\frac{L_s}{\lambda_0}$ at fixed $\Ra$}. They explain this phenomena as due to the reduction of the vertical Reynolds number: as $\frac{L_s}{\lambda_0}$ increases, the zonal flow becomes stronger, and it suppresses the vertical velocity, leading to a decrease in $\Nu$. In this regime their DNS indicate the scaling $\Nu\sim \Ra^{0.16}$. 
As it can be noticed in Table \ref{table:overview_of_the_results}, our rigorous upper bounds indicate a behaviour in agreement with the results in \cite{huang2022effect}: in the turbulent regime, as the slip length $L_s$ increases, the heat transport decreases. In fact, observe that, in each $\Pr$ region, the heat transport increases from 
$\Nu\lesssim \Ra^{\frac{5}{12}}$ when $L_s=\infty$ to $\Nu\lesssim  L_s^{-\frac 23}\Pr^{-\frac 16}\Ra^{\frac{1}{2}}$ when $L_s<1$ (see Table \ref{table:overview_of_the_results} where the upper bounds are ordered from the smallest to the largest, in each $Pr$-region). Let us recall that for no-slip boundary conditions ($L_s=0$) Otto, Choffrut and the second author of this paper find $\Nu\lesssim (\Ra\ln(\Ra))^{\frac 13}$ when $\Pr\gtrsim (\Ra\ln(\Ra))^{\frac 13}$ and $\Nu\lesssim \Pr^{-\frac 12}(\Ra\ln(\Ra))^{\frac 12}$ when $\Pr\lesssim (\Ra\ln(\Ra))^{\frac 13}$ \cite{choffrutNobiliOttoUpperBoundsOnNusseltNumberAtFinitePrandtlNumber}, while Doering and Constantin in  \cite{doering1996variational} prove $\Nu\lesssim \Ra^{\frac{1}{2}}$ uniformly in $\Pr$. In the regime of small $L_s$ our upper bound is $\Nu\lesssim L_s^{-\frac 23}\Pr^{-\frac 16}\Ra^{\frac12}$ for \textit{any} Prandtl number. We cannot directly compare the results in the present paper with the results available for no-slip boundary conditions in \cite{choffrutNobiliOttoUpperBoundsOnNusseltNumberAtFinitePrandtlNumber, doering1996variational} since our analysis brakes down in the limit $L_s\rightarrow 0$. 
With respect to the long-standing open problem regarding the scaling laws for the Nusselt number in the "ultimate state" \cite{zhu2018transition, doering2019absence, zhu2019zhu, doering2020absence} we can say the following: 
on the one hand, the DNS in \cite{huang2022effect} and our analysis seem to suggest that, for any Prandtl number region and any $L_s\neq 0$, the Malkus's scaling $\Nu\sim\Ra^{\frac{1}{3}}$ is not achievable in the \textit{turbulent regime}. This is quite evident when looking at the results in Table \ref{table:overview_of_the_results}, since (as mentioned before) the upper bounds are ordered and increase from $\Ra^{\frac{5}{12}}$ as $L_s$ decreases. However we warn that the reality of heat transport might be much more complicated than this nice picture. 
With our arguments we are only able to derive \textit{upper bounds} and cannot exclude the possibility of smaller scaling exponents. Indeed one of the biggest theoretical challenges for this problem is to find (non trivial) lower bounds for the Nusselt number. Even just exhibiting solutions that attain certain upper bounds in some simplified situation (in the spirit of \cite{souza2015maximal}) would be a big step ahead in the theory. Interesting ideas in this direction have been developed by Tobasco and Doering: in \cite{tobasco2017optimal, tobasco2022},  inspired by problems arising in the study of energy-driven pattern formation in materials science, the authors design a two dimensional "branching" flow that  transports at rate $\Nu\sim {\rm{Pe}}^{\frac 23}$ (up to a logarithmic correction) for a passive tracer which diffuses and is advected by a divergence-free velocity field with a fixed enstrophy budget $\langle\|\nabla \mathbf{u}\|_{L^2}^2\rangle\sim {\rm{Pe}}^2$. Here ${\rm{Pe}}$ is the Peclet number. The remarkable lower bound they derive using "branching" techniques is responsible for the logarithmic corrections. Although their result in particular shows that the $\Ra^{\frac 12}$ "ultimate" scaling is attained by flows that do not solve  \eqref{navier-stokes}--\eqref{eq:T}, these arguments can be used to detect mechanisms that make the scalings $\rm{Ra}^{\frac 12}$ and  $\rm{Ra}^{\frac 13}$ realizable for buoyancy driven solutions.

\begin{figure}
    \centerline{\includegraphics[scale=0.06]{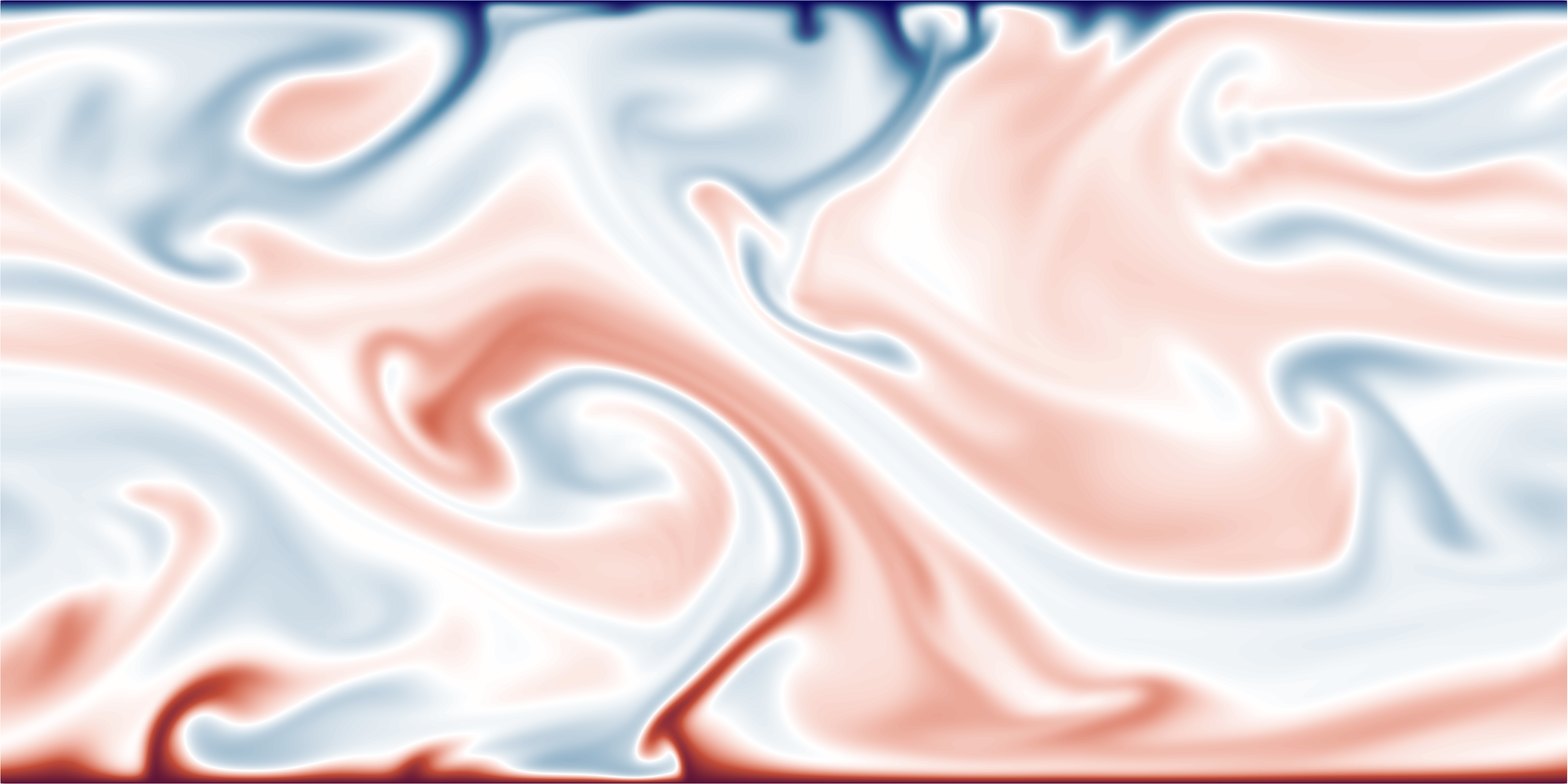}}
    \caption{Thermal boundary layers arising at $\Pra=1$, $L_s=1$ and $\Ra=10^8$. This snapshot is generated by a simulation.}
    \label{fig:simulation}
\end{figure}

Finally we want to conclude with a remark.
Our upper bounds on the Nusselt number quantify heat transport properties of the (thermal) boundary layer without characterizing it. The only information we can extrapolate from our analysis is about the (thermal) boundary layer thickness (see Figure \ref{fig:simulation}).
On the other hand, in the excellent work in \cite{gie2019boundary}, Gie and Whithehead are able to describe the flow in the boundary layer. In a $3D$ periodic channel, they study the Boussinesq system at very small viscosities and with Navier-slip boundary conditions. Through the explicit construction of a corrector, they show that the behaviour of the dynamics in the boundary layers is governed by the Prandtl equations. In particular these equations can be linearized for Navier-slip boundary conditions, implying that, in this specific case, the boundary layers are non-turbulent. From this the authors deduce that the "ultimate" state, which is based on the hypothesis of a turbulent boundary layer \cite{ahlers2009heat}, can not exist in this particular setup.}

\section*{Aknowledgements} The authors acknowledge the support by the Deutsche Forschungsgemeinschaft (DFG) within the Research Training Group GRK 2583 "Modeling, Simulation and Optimization of Fluid Dynamic Applications”.

\section*{Declaration of Interests} 
The authors report no conflict of interest.

\bibliographystyle{jfm}
\bibliography{bibliography}
\end{document}